\documentclass[12pt]{article}
\usepackage{mathtools}
\usepackage{enumitem}
\usepackage{amsmath, amssymb,amsthm}
\usepackage[margin=1in]{geometry}
\usepackage[utf8]{inputenc}
\usepackage{tikz}
\usetikzlibrary{arrows,chains,matrix,positioning,scopes}

\newtheorem{theorem}{Theorem}[section]
\newtheorem{proposition}[theorem]{Proposition}
\newtheorem{lemma}[theorem]{Lemma}
\theoremstyle{definition}
\newtheorem{definition}[theorem]{Definition}

\begin{document}

\title{Factorizations of Surjective Maps of Connected Quandles}

\author{T. Braun, C. Crotwell, A. Liu, P. Weston, \&  D.N. Yetter }

\maketitle

\begin{abstract}
    We consider the problem of when one quandle homomorphism will factor through another, restricting our attention to the case where all quandles involved are connected.  We provide a complete solution to the problem for surjective quandle homomorphisms using the structure theorem for connected quandles of Ehrman et al. \cite{Ehrman} and the factorization system for surjective quandle homomorphsims of Bunch et al. \cite{Bunch} as our primary tools.  The paper contains the substantive results obtained by an REU research group consisting of the first four authors under the mentorship of the fifth, and was supported by National Science Foundation, grant DMS-1659123.
\end{abstract}

\section{Motivation}

Quandles were originally defined by Joyce \cite{Joyce} for the purpose of studying classical knots.  However, as was shown by Yetter \cite{Yetter}, they are closely bound up with monodromy phenomena. They allow a purely algebraic description of monodromy as a homomorphism of quandles, or of augmented quandles, from the knot quandle of the singular set in the base (possibly equipped with its canonical augmentation in the fundamental group of the complement of the singular set) to a quandle (or augmented quandle) associated with the generic fiber.  Thus the study of homomorphisms of quandles has the potential for application in a variety of geometric settings in which monodromy naturally arises (as for example, branched coverings and Lefschetz fibrations). 

The study of the structure of quandle homormorphisms is intrinsically difficult; for in our present understanding, there is very little control over how subquandles sit inside larger ambient quandles (cf. Ehrman, et al. \cite{Ehrman} and Nelson and Wong \cite{Nelson} which give decomposition theorems and corresponding constructions of larger quandles from subquandles which are orbits under the inner automorphism group). 

The present work is a contribution to the difficult problem of determining conditions under which one quandle homomorphism will factor through another:  a problem, the solution to which would allow the construction of invariants of branched covers of 3-manifolds from quandle colorings of the link which is the branch set of the cover.  

Here, we consider only those homomorphisms with a(n algebraically) connected source: these would describe monodromy in which the singular set in the base is topologically connected (as in the case of a branched covering with a knot as the branch set).  We, however, do not address the full problem, restricting our attention to a setting in which both the quandles and the homomorphisms involved are well understood:  surjective homomorphisms between connected quandles.
This restriction is less onerous than it might seem, as the image of any quandle homomorphism with connected source is always a connected subquandle of the target quandle.

\section{Background}
We recall from Joyce \cite{Joyce}

\begin{definition} A \textit{quandle} is a set $Q$ with two binary operations, $\rhd$ and $\rhd^{-1}$ such that
\begin{enumerate}[label=\roman*)]
\item $\forall x\in Q$, $ x\rhd x=x$
\item $\rhd$ is right-invertible, i.e. $\forall x,y\in Q,$  $(x\rhd y) \rhd^{-1}y=x=(x\rhd^{-1}y)\rhd y$
\item $\rhd$ is right-distributive, i.e. $\forall x,y,z \in Q$, $(x\rhd y)\rhd z= (x\rhd z) \rhd (y\rhd z)$
\end{enumerate}

A \textit{quandle homomorphism} is a mapping $\Phi$ between $(Q,\rhd)$ and $(Q',\rhd')$ such that $\forall x,y\in Q$, $\Phi(x\rhd y)=\Phi(x)\rhd'\Phi(y)$. Preservation of $\rhd$ implies preservation of $\rhd ^{-1}$.

\end{definition}

Much of the complication in quandle theory comes from the lack of constants in the theory, resulting in there being no analogues of the kernels found in group and ring theory. Fortunately, quandle theory is initimately related to the more familiar theory of groups in a number of ways.  The first was observed by Joyce \cite{Joyce}: by defining $x\rhd y$ as $xyx^{-1}$ and $x\rhd^{-1} y$ as $x^{-1}yx$, any union of conjugacy classes in a group $G$ forms a quandle. It is a theorem of Joyce \cite{Joyce} that free quandles admit such a representation (as a union of conjugacy classes in a free group). When the entire group $G$ is made into a quandle in this way, it is denoted $Conj(G)$, and group homomorphisms induce quandle homomorphisms, so that $Conj$ is a functor from the category of groups, {\bf Grps}, to the category of quandles, {\bf Quand}.

There are also three groups canonically associated to any quandle.

First, like any mathematical object, quandles admit an \textit{automorphism group}:  $Aut(Q)$ being the group of invertible quandle homomorphisms from $Q$ to itself.

Second, the \textit{inner automorphism group} of a quandle, denoted $Inn(Q)$, is the subgroup of $Aut(Q)$ generated by the symmetries $S_x$ of Q: that is, the maps given by 
 \[S_x(y)\coloneqq y\rhd x .\]
 
 \noindent It is an easy exercise to show that $Inn(Q)$ is abelian if and only if it is trivial.
 
Finally, there is a group $AdConj(Q)$ with presentation
\[ \langle\; q \in Q\; |\; (q\rhd r)rq^{-1}r^{-1}, (q\rhd^{-1} r)r^{-1}q^{-1}r \;\; \mbox{\rm for all $q,r \in Q$}\rangle\]

\noindent The notation derives from the theorem of Joyce that $AdConj$ is the left adjoint to the functor from $Conj:{\bf Grps}\rightarrow {\bf Quand}$.

A \textit{connected quandle} is a quandle $Q$ that contains only one orbit under the action of $Inn(Q)$.

As our primary object of study will be surjective quandle homomorphisms,
we should recall several results and definitions from Bunch et al. \cite{Bunch}.

First

\begin{theorem} \cite{Bunch} \label{epifunctor}
The assignment of inner automorphism groups to quandles is a functor from
${\bf Quand_{epi}}$ the category with quandles as objects and surjective quandle homomorphisms as arrows to ${\bf Grps}$ the category of groups, and 

\[ Q \mapsto (Q, Inn(Q), q\mapsto -\rhd q) \]

\noindent is a functor from ${\bf Quand_{epi}}$ to ${\bf AugQuand}$, the category of augmented quandles.
\end{theorem}

\begin{definition} \cite{Bunch}
A \textit{rigid quotient} is a surjective quandle homomorphism $h: Q\twoheadrightarrow R$ such that the induced group homomorphism $f: Inn(Q)\rightarrow Inn(R)$ is an isomorphism.
\end{definition}

\begin{definition} \cite{Bunch}
A \textit{realizable kernel} is a subgroup $H \subset Inn(Q)$ where $h$ is a surjective quandle homomorphism $h:Q \twoheadrightarrow R$ and $H = ker(Inn(h))$.
\end{definition}


It should be observed first that this definition is a property of normal subgroups of an inner automorphism group $Inn(Q)$ of a specified quandle $Q$, not of an abstract group isomorphic to $Inn(Q)$.  And second that it is non-vacuous:  unless the commutator subgroup is the entire group, it is never a realizable kernel.

Curiously, the orbits under any normal subgroup of $Inn(Q)$ acting on $Q$ admit a quandle structure induced by that on $Q$, indeed Bunch et al. \cite{Bunch} showed that the orbits under a subgroup of $Inn(Q)$ have an induced quandle structure if and only if the subgroup is normal.  But if $N$ is not a realizable kernel, then there will be a larger subgroup which is a realizable kernel, and which has the same orbits as $N$:

Bunch et al. \cite{Bunch} also showed that the intersection of an arbitrary family of realizable kernels in $Inn(Q)$ is a realizable kernel, and thus (noting that the whole group is a realizable kernel), there is a closure operation on the normal subgroups of $Inn(Q)$, giving the smallest realizable kernel containing $N$.  We denote the closure of $N$ under this operation by $N^Q$, as we will have cause to consider the closure operation with respect to different quandles with the same inner automorphism group simultaneously.

Bunch et al. \cite{Bunch} also provided a constructive description of this closure operation:  

\begin{theorem} \cite{Bunch} 
Let $N$ be a normal subgroup of $G = Inn(Q)$, and let $[q]$ denote the orbit of $q$ under the restriction of the action of $G$ to $N$, and $c_N:Q \rightarrow Q/N$ the quotient quandle homomorphism, then

\[ N^Q = ker(Inn(c_N)) = \bigcap_{X \in Q} G_{[x]} \].
\end{theorem}

\begin{theorem} \cite{Bunch} \label{or-factorization}
For any surjective quandle homomorphism $h:Q\twoheadrightarrow R$, if $N=ker(Inn(h))$ and $g_N$ is the canonical homomorphism from $Q$ to $Q/N$ then $h$ can be factored as $f(g_N)$, where $f:G/N\twoheadrightarrow R$ is a rigid quotient.
\end{theorem}

Here $Q/N$ is the set of orbits of $Q$ under the restriction of the action of $Inn(Q)$ on $Q$ to $N$ with the quandle structure $[q]\rhd [r] = [q\rhd r]$ with $[x]$ denoting the orbit of $x \in Q$.  We refer to such a quotient as an \textit{orbit quotient}.

We also observe that the proof of several results on connected quandles given by Ehrman et al. \cite{Ehrman} will carry a similar result without restricting the cardinality of the quandle in question to be finite:

In what follows, let $Q$ be a connected quandle and let $G=Inn(Q)$.  

We then have, by a well-know elementary result on group actions

\begin{proposition} \label{homogeneous}
There is an equivariant bijection between $Q$ and $H\backslash G$, the set of right cosets of $H$.
\end{proposition}


\begin{theorem} \label{connstructure}
Fix an element $q\in Q$, let $H < G$ be its stabilizer, and identify $Q$ with the set of right $H$ cosets, $H\backslash Q$ by the equivariant map of Proposition \ref{homogeneous}. 
Let $|\cdot |$ be the augmentation map $|\cdot|:H\backslash G \rightarrow G$ that is the map which assigns to $Hg$ the map $(- \rhd Hg) \in G$. Then letting $\eta =|H|$ we have $\eta \in Z(H)$ where $Z(H)$ is the center of $H$. Further, $G$ is generated by $\{ g^{-1}\eta g \; | \; g\in G\}$ and 
the action of $G$ on $H\backslash G$ is faithful. Moreover, the quandle operation is given in terms of stabilizer cosets by $Hg \rhd H\gamma = Hg\gamma^{-1}\eta \gamma$.
\end{theorem}

\begin{proof}
The proof is essentially identical to the proof of Ehrman et al. \cite{Ehrman} Theorem 4.3.  The faithfulness of the action is immediate from the fact the $G$ is a subgroup of $Aut(Q)$.
\end{proof}

Similary omitting cardinality restrictions, the proof of Ehrman et al. \cite{Ehrman} Theorem 4.5 gives

\begin{theorem} \label{connconstruct}
Let $H < G$ be a group with specified subgroup $H$, and $\eta \in Z(H)$ and element of the center of the subgroup such that $\{ g^{-1}\eta g \; | \; g \in G \}$ generates $G$. Then $H\backslash G$ is a connected quandle under the operation $Hg \rhd H\gamma : = Hg\gamma^{-1}\eta \gamma$.  If, moreover, the action of $G$ on $H\backslash G$ is faithful, then $Inn(H\backslash G) = G$, otherwise $Inn(H\backslash G) \cong G/N_{H\backslash G}$ where 
\[ N_{H\backslash G} := \{ g\in G \; |\; \forall \gamma \in G H\gamma = H\gamma g \; \} \]
\end{theorem}

\section{Results}

Using the definitions and results of the preceding section as our primary tools, we now consider the question of when one surjective quandle homomorphism between connected quandles, $g:Q\rightarrow R_2$,
factors through another $h:Q\rightarrow R_1$ in the sense that there is a quandle homomorphism $\phi:R_1\rightarrow R_2$ such that $g = \phi(h)$.

In view of Theorem \ref{or-factorization} the quandle homomorphisms $g$ and $h$ factor as shown by the solid arrows of the diagram below
 
$$\begin{tikzpicture}
  \matrix (m) [matrix of math nodes,row sep=1em,column sep=1em,minimum width=0.5em]
  {
    R_{1} \\
     Q/N_{1}  \\
     Q & Q/N_{2} & R_{2} \\};
     
  \path[-stealth]
  (m-2-1) edge node [left] {} (m-1-1)
         edge [dashed] node [below] {}(m-3-2)
 (m-3-1) edge node [left]{}(m-2-1)
       edge node [below] {} (m-3-2)
(m-3-2) edge node [below]{} (m-3-3)
(m-1-1) edge  [dashed] node [below]{} (m-3-3);
\end{tikzpicture}$$

We give conditions for the existence of a homomorphism between $Q/N_{1}$ and $Q/N_{2}$, as shown in the diagram, then extend the problem to rigid quotients, providing conditions for a homomorphism from $R_{1}$ to $R_{2}$. 

Our first result does not require connectedness as a hypothesis.  Its proof turns largely on the following elementary observation about group actions, wherein the orbit of $x_i$ under the action of $G$ is denoted $x_iG$.

\begin{lemma} \label{sameorbit} Suppose the group $G$ acts on the set $X$. Let $x_1, x_2 \in X$. Then $x_1G=x_2G$ iff $x_1=x_2 \cdot g$ for some $g \in G$.
\end{lemma}



Consider the following:

\begin{theorem}
Let $g_{N_1}: Q \to Q/N_1$ and $g_{N_2}: Q \to Q/N_2$ be two orbit quotients of $Q$ by realizable kernels $N_i$ ($i = 1, 2$).  Denoting the orbit of $q\in Q$ under the action of $N_i$ by $q\cdot N_i$, these are given by $g_{N_i}(q) = q\cdot N_i$. Then there is a quandle homomorphism $\Omega : Q/N_1\to Q/N_2$ such that the diagram
$$\begin{tikzpicture}[node distance=2cm, auto]
  \node (C) {$Q/N_1$};
  \node (P) [below of=C] {$Q$};
  \node (Q) [right of=P] {$Q/N_2$};
  \draw[dashed, ->] (C) to node {$\Omega$} (Q);
  \draw[->>] (P) to node {$g_{N_1}$} (C);
  \draw[->>] (P) to node [swap] {$g_{N_2}$} (Q);
\end{tikzpicture}$$

\noindent commutes if and only if $N_1 \subseteq N_2$.

\end{theorem}

\begin{proof}
Plainly for the diagram to commute, $\Omega : Q/N_1\rightarrow Q/N_2$ must be given by $\Omega(q \cdot N_1) = q \cdot N_2$. The proof of the theorem thus consists in showing that this formula give a well-defined quandle homomorphism whenever $N_1 \subseteq N_2$, and conversely.

First, we will prove that $\Omega$ is well-defined. Let $a \cdot N_1, b \cdot N_1 \in Q/N_1$. Then $\Omega(a \cdot N_1)=a \cdot N_2$ and $\Omega(b \cdot N_1)=b \cdot N_2$. Suppose $a \cdot N_1=b \cdot N_1$. Then by Lemma \ref{sameorbit} we have $a = b \cdot n$ for some $n \in N_1$. Since $N_1 \subset N_2$, we have  $n \in N_2$ so  $a = b \cdot n$ for some $n \in N_2$. By Lemma \ref{sameorbit} we then have $a \cdot N_2=b \cdot N_2$. Since $\Omega(a \cdot N_1)=a \cdot N_2$ and $\Omega(b \cdot N_1)=b \cdot N_2$. Thus, we have $\Omega(a \cdot N_1)=\Omega(b \cdot N_1)$.

Let the orbit $q \cdot N_1, q' \cdot N_1 \in Q/{N_1}$. Then $\Omega(q \cdot N_1) = q\cdot N_2$ and $\Omega(q' \cdot N_1)=q' \cdot N_2$.  Since $g_{N_{2}}$ is a quandle homomorphism, $q \cdot N_2 \rhd q' \cdot N_2 = (q  \rhd q') \cdot N_2$, Therefore,
$$
\Omega(q \cdot N_1) \rhd \Omega(q' \cdot N_1) = q \cdot N_2 \rhd q' \cdot N_2 = (q \rhd q') \cdot N_2 = \Omega((q \rhd q') \cdot N_1).
$$ 
Hence, $\Omega$ is a homomorphism.

Suppose $\Omega$ is a well-defined homomorphism. Assume, for the purpose of contradiction, that $N_1 \nsubseteq N_2$. Let $x \cdot N_1 \in Q/N_1$ where $x \in N_1$ but $x \notin N_2$. Then $\Omega(xN_1)=xN_2$. But $\Omega(x \cdot N_1)=\Omega(N_1)=N_2$, which implies $x\cdot N_2=N_2$ and $x \in N_2$. Thus, by contradiction, $N_1 \subseteq N_2$.
\end{proof}

The previous result did not depend on the quandles involve being connected. 
We now turn to results in which connectededness is important, beginning with characterizing orbit quotients of connected quandles in terms of the presentation as right cosets of a stablilizer subgroup given in Theorem \ref{connstructure}.  Again, our result requires a lemma from the theory of group actions:

\begin{lemma} Let $S$ be a set equipped with a right action of a group $\Gamma$, $f: S\times\Gamma\rightarrow S$ and $N$ a normal subgroup of a group $\Gamma$. Then the map $f_N: S/N \times \Gamma/N \rightarrow S/N$ defined by $f_N(s\cdot N, gN)=(s \cdot g) \cdot N$ is a well-defined group action.  And, moreover, when this action is lifted to an action of $\Gamma$ along the quotient map $c:\Gamma \rightarrow \Gamma/ N$, the map $S \rightarrow S/N$ carrying each element to its orbit under the restriction of the action to $N$ is $\Gamma$-equivariant, that is, the diagram
$$\begin{tikzpicture}
  \matrix (m) [matrix of math nodes,row sep=3em,column sep=4em,minimum width=2em] {
     S \times \Gamma & & S\\
     S/N \times \Gamma & S/N \times \Gamma/N & S/N \\};
  \path[-stealth]
    (m-1-1) edge node [above] {$f$} (m-1-3)
            edge node [left] {} (m-2-1)
    (m-2-2) edge node [below] {$f_N$} (m-2-3)
    (m-1-3) edge node [right] {} (m-2-3)
    (m-2-1) edge node [below] {$Id \times c$} (m-2-2);
\end{tikzpicture}
$$

\noindent commutes.
\end{lemma}

\begin{proof}
We denote the original group action $f$ by the infix $\cdot$.

\setlength\parindent{36pt}\textbf{Well-definition:}
Suppose for some $\sigma,\delta\in S$ and $g,h \in\Gamma$, $(\sigma\cdot N,gN)=(\delta\cdot N,hN)$, which implies $\sigma\cdot N=\delta\cdot N$ and $gN=hN$. Hence, by Lemma \ref{sameorbit}, $\sigma =\delta\cdot n$ and $g=h\cdot n'$ where $n, n' \in N$. Therefore, $f_N(\sigma\cdot N,gN) = (\sigma\cdot g)\cdot N = (\delta\cdot n\cdot h\cdot n')\cdot N=f_N((\delta\cdot\ n)\cdot N, (h n')N)= f_N(\delta\cdot N,hN)$. Thus $f_N(\sigma\cdot N,gN)=f_N(\delta\cdot N,hN)$.

\setlength\parindent{36pt}\textbf{Identity Property:}
Let $e$ be the identity element of $\Gamma$. Then $f_N(\sigma\cdot N,eN) = (\sigma\cdot e)\cdot N=\sigma\cdot(eN)=\sigma\cdot N$.

\setlength\parindent{36pt}\textbf{Associativity:}
We have $f(f(\sigma\cdot N,gN),hN)=f((\sigma\cdot g)\cdot N,hN)=((\sigma\cdot g)\cdot h)\cdot N= (\sigma\cdot (gh))\cdot N= f(\sigma\cdot N, ghN)$.

\setlength\parindent{36pt}\textbf{Equivariance:}
Having shown that $f_N$ is well-defined, this is immediate by construction.
\end{proof}

Now, let $Q$ be a connected quandle, $G=Inn(Q)$, $q\in Q$ and $H = G_q$ the stabilizer subgroup. As in Theorem \ref{connstructure}, we can identify $Q$ with $H \backslash G$ and give the quandle operation in terms of the augmentation value $\eta \in Z(H)$ of the trivial coset $H$.

\begin{theorem}
For $Q, G, q, H$ and $\eta$ as above, let $N$ be a realizable kernel, then the orbit quotient quandle $Q/N$ whose elements are the orbits of $Q$ under the restriction of the action of $G$ to $N$ is connected with $Inn(Q/N) \cong G/N$. 
Let $H_1:=(G/N)_{qN} =\{hN \mid qN \cdot hN = qN\}$ be the stabilizer of the image of $q$ in $Q/N$. Denoting the augmentation for either quandle by $| - |$, we then have
\begin{itemize}
    \item $H_1=HN/N$,
    \item $|H_1|=\eta N$,
    \item $|H_1|=\eta N \in Z(H_1)$,
    \item $|H_1gN|= g^{-1}\eta gN$,
    \item $G/N$ is generated by $\{ |H_1gN|\; | \; g\in G \}$
\end{itemize}
\end{theorem}

\begin{proof}
For the first item, we begin by showing that $H_1 \subseteq HN/N$. Let $gN \in H_1$. That is, $qN \cdot gN = qN$. Hence, we must show that $g \in HN$. By Lemma \ref{connstructure}, $qN \cdot gN=(q \cdot g) \cdot N = q \cdot N$, which implies there exists $n \in N$ such that $q \cdot g \cdot n = q$. Therefore, $g \cdot n \in H$, so $g \in HN$. Thus, $H_1 \subseteq HN/N$.

Next, we will show $HN/N \subseteq H_1$. Let $hN \in HN/N$. We must show that $qN \cdot hN = qN$. Let $q \cdot h \cdot n \in q \cdot h \cdot N$. Then since $q \cdot h = q$, $q \cdot h \cdot n = q \cdot n \in q \cdot N$, so $q \cdot h \cdot N = q \cdot N$. Hence, by Lemma \ref{connstructure}, $qN \cdot hN = qN$. Therefore, $hN \subseteq H_1$, so $HN/N \subseteq H_1$ Thus, $H_1=HN/N$.

The second item follows from the functoriality of the augmented quandle structure of Theorem \ref{epifunctor}.

For the third, let $gN \in H_{1}$. Then we must show that $\eta N gN=gN \eta N$. Following Property 1, $|H_1| gN = \eta N gN$. Since $\eta \in Z(H) \subseteq G$ and $N \lhd G$, $\eta N gN = N \eta gN$. Hence $|H_1| gN = \eta N gN = N \eta g N = N g \eta N = gN \eta N = gN |H_1|$.

For the fourth,  following Property 1, we have $|H_1gN|= g^{-1}N\eta NgN=g^{-1}\eta N N gN=g^{-1}\eta NgN$ since $\eta N \in Z(H_1)$. Hence, because $N \lhd G$, $g^{-1}\eta N g N=g^{-1} \eta g N N=g^{-1} \eta g N$.

And finally, since $g^{-1}\eta g$ generates $G$, so by Property 3, $g^{-1}\eta gN=|H_1gN|$ generates $G/N$.
\end{proof}

    

Our next results characterize rigid quotients of connected quandles in a manner analogous to the way in which Theorems \ref{connstructure} and \ref{connconstruct} characterized connected quandles. 

\begin{proposition} \label{easyprop}
Let $c:Q\rightarrow R$ be a rigid quotient map with $Q$ a connected quandle, $q \in Q$, and $G = Inn(Q) = Inn(R)$.  Then $R$ is connected and letting $H < G$ (resp. $K < G$) be the stabilizer of $q$ (resp. $c(q)$), and identifying $Q$ with $H\backslash G$ and $R$ with $K\backslash G$ as in Theorem \ref{connstructure}, the following hold:

\begin{itemize}
    \item $K < H$
    \item The augmentation values of the trivial cosets $H$ and $K$ are equal and denoting this element of $G$ by $\eta$, we have $\eta \in Z(K) \cap H$.
    \item The quandle operations on both quandles are given in terms of cosets by the formulas of Theorem \ref{connstructure}.
    \item $G$ acts faithfully on both $H\backslash G$ and $K\backslash G$
\end{itemize}
\end{proposition}

\begin{proof}
First observe that a rigid quotient is necessarily a $G$-equivariant map.  From this it follows that the stabilizer of an element must be contained in the stablilizer of its image.  The third and fourth conclusions are immediate consequences of Theorem \ref{connstructure}, so it suffices to show the second.  The entire content of the second statement is that the two augmentation values are equal, since once this is established, it follows that $\eta$ lies in both $Z(K)$ and $Z(H) < H$ by Theorem \ref{connstructure}.  But this follows immediately from the functoriality of augmentations in $Inn(-)$ with respect to surjective quandle homomorphisms.
\end{proof}

Having characterized both rigid quotients and orbit quotients of connected quandles in terms of their presentation via stabilizer cosets, we turn at last to the question of when the upper square in our original diagram can be completed so that the one given quotient factors through the other. As the homomorphism from $Q/N_1$ to $Q/N_2$ is itself an orbit quotient, we simplify the situation to considering when a solid diagram of the form,

$$\begin{tikzpicture}
  \matrix (m) [matrix of math nodes,row sep=3em,column sep=4em,minimum width=2em] {
     R & R^\prime  \\
      Q & Q/N\\};
  \path[draw, ->>]
    (m-1-1) edge [dashed] node [above] {} (m-1-2)
    (m-2-2) edge node [right] {$c^\prime$} (m-1-2)
    (m-2-1) edge node [left] {$c$} (m-1-1)
    (m-2-1) edge node [below] {$g_N$} (m-2-2);
\end{tikzpicture}
$$

\noindent where $c$ and $c^\prime$ are rigid quotients and $g_N$ is an orbit quotient, admits the existence of the quandle homomorphism indicated by the dotted arrow to complete a commutative square.  Without loss of generality we will assume that $N$ is a realizable kernel, so that $Inn(Q/N) = Inn(R^\prime) = G/N$.

Now if the dotted arrow exists, it must factor as an orbit quotient, followed by a rigid quotient, giving rise to the dotted portion of the diagram below, in which $c_N$ must be a rigid quotient.

$$\begin{tikzpicture}
  \matrix (m) [matrix of math nodes,row sep=3em,column sep=2em,minimum width=0.5em]
  {
    K\backslash G & (K\backslash G) / N & R \\
     H\backslash G & (H\backslash G)/N  \\
     };
     
    \path[-stealth]
    (m-1-1) edge [dashed] node [below]{} (m-1-2)
      (m-2-1)  edge node [left]{$c$} (m-1-1)
                edge  node [below]{$g_N$} (m-2-2)
    (m-1-2) edge [dashed] node [below]{} (m-1-3)
    (m-2-2) edge [dashed] node [left]{$c_N$} (m-1-2)
            edge node [right]{$c'$} (m-1-3);

\end{tikzpicture}$$

The existence of the dotted arrow thus reduces to the condition that the induced map between the orbit quotients of $Q = H\backslash G$ and $R = K\backslash G$ be a rigid quotient and a condition for one rigid quotient of connected quandles to factor through another.  The first of these is given by




\begin{theorem} In the diagram above,
$c_N:(H\backslash G ) /N \twoheadrightarrow (K\backslash G )/N $ is a rigid quotient iff $N=N^{K\backslash G}$.
\end{theorem}

\begin{proof}

Suppose $c$ is a rigid quotient. Then $Inn((H\backslash G)/N) = Inn(H\backslash G)/N = G/N$.  $Inn((H\backslash G)/N) = G/N =  Inn((K\backslash G)/N)$. Then $Inn(K\backslash G)/N = Inn(H\backslash G)/N $ so $N = N^{K\backslash G}$.

Suppose $N=N^{K\backslash G}$. Then $Inn(K\backslash G)/N^{K\backslash G} = G/N^{K\backslash G} = G/N = Inn(K\backslash G)/N$ so $Inn(K\backslash G)/N = Inn(H\backslash G)/N = G/N$.


\end{proof}

Observe that if we had not assumed $N$ was a realizable kernel, the condition would become $N^{H\backslash G} = N^{K\backslash G}$.

If this condition holds, the existence of the dotted map to complete the square formed by an orbit quotient and two given rigid quotients reduces to the question of when one rigid quotient factors through another.  Using the characterization of rigid quotients of connected quandles given by Proposition \ref{easyprop}, the necessary and sufficient condition is given by containment of stabilizer subgroups:

\begin{theorem}
Let $c: H\backslash G \to K\backslash G$ be defined by $c(Hg)=Kg$ and $c': H\backslash G \to K\backslash G$ by $c'(Hg)=Lg$ be rigid quotients, in both cases with the quandle structure induced by an element $\eta$ which must lie in both $Z(K)\cap H$ and $Z(L)\cap H$ by Proposition \ref{easyprop}. Then $\Phi : K\backslash G \to L\backslash G$ defined by $\Phi (Kg)=Lg$ is a well-defined quandle homomorphism if and only if  $K \subset L$, and moreover is a rigid quotient and the diagram
$$\begin{tikzpicture}[node distance=2cm, auto]
  \node (C) {$K\backslash G$};
  \node (P) [below of=C] {$H\backslash G$};
  \node (Q) [right of=P] {$L\backslash G$};
  \draw[->] (C) to node {$\Phi$} (Q);
  \draw[->] (P) to node {$c$} (C);
  \draw[->] (P) to node [swap] {$c'$} (Q);
\end{tikzpicture}$$

\noindent commutes.
\end{theorem}


\begin{proof}
Let $Hg_1, Hg_2 \in H\backslash G$. Since $c$ is a homomorphism, $c(Hg_{1} \rhd Hg_{2})=c(Hg_{1}) \rhd c(Hg_{2})$, which implies $Kg_{1}g_{2}^{-1}|H|g_{2}=Kg_{1}g_{2}^{-1}|K|g_{2}$. Hence, since $G$ acts faithfully on $K\backslash G$, $|H|=|K|$. Proceeding similarly, $c'(Hg_{1} \rhd Hg_{2})=c'(Hg_{1}) \rhd c'(Hg_{2})$, which implies $Lg_{1}g_{2}^{-1}|H|g_{2}=Lg_{1}g_{2}^{-1}|L|g_{2}$. Therefore, $|H|=|L|$. Thus $|K|=|H|=|L|$.

\setlength\parindent{36pt}\textbf{Well-defined:} Define $\Phi : K\backslash G \to L\backslash G$ by $\Phi (Kg)=Lg$. First, we will prove that $\Phi$ is well-defined. Let $Ka, Kb \in K\backslash G$. Then $\Phi(Ka)=La$ and $\Phi(Kb)=Lb$. Suppose $Ka=Kb$. Then by $Lemma$ we have $a = k \cdot b$ for some $k \in K$. Since $K \subset L$, which implies  $k \in L$, $La=Lb$. Thus, we have $\Phi(Ka)=La=Lb=\Phi(Kb)$.

\setlength\parindent{36pt}\textbf{Homomorphism:} Let $Kg_{1}, Kg_{2} \in K\backslash G$. Then $\Phi(Kg_{1})=Lg_{1}$ and $\Phi (Kg_{2})=Lg_{2}$. 
We have $\Phi (Kg_{1} \rhd Kg_{2})=\Phi (Kg_{1}g_{2}^{-1}|K|g_{2})=Lg_{1}g_{2}^{-1}|K|g_{2}$ and $\Phi (Kg_{1}) \rhd \Phi(Kg_{2})=Lg_{1}\rhd Lg_{2}=Lg_{1}g_{2}^{-1}|L|g_{2}$. Since $|K|=|L|$, we see $Lg_{1}g_{2}^{-1}|K|g_{2}=Lg_{1}g_{2}^{-1}|L|g_{2}$.
Hence, $\Phi$ is a homomorphism.

\setlength\parindent{36pt}$\boldsymbol{K \subset L:}$ Suppose $\Phi$ is a well-defined homomorphism. Assume, for the purpose of contradiction, that $K \nsubseteq L$. Let $Ka \in K\backslash G$ where $a \in K$ but $a \notin L$. Then $\Phi(Ka)=La$. But $\Phi(Ka)=\Phi(K)=L$, which implies $La=L$ and $a \in L$. Thus, by contradiction, $K \subset L$.

\setlength\parindent{36pt}\textbf{Rigidity:} This follows immediately from Proposition \ref{easyprop}.

\setlength\parindent{36pt}\textbf{Commutativity:} Having shown that $\Phi$ is well-defined, this is immediate by construction.
\end{proof}

\section{Directions for Future Research}

The present paper has given a complete solution to the problem of when one surjective quandle homomorphism factors through another in the case where the source and target quandles of both maps are connected.  The fully general problem, in which both simplifying hypotheses, connectedness and surjectivity, are dropped, appears intractable at the moment.  However, the hypotheses can be weakened bit by bit.  

Dropping surjectivity raises the question of how to understand inclusions of connected quandles into other connected quandles.  An adequate understanding of such inclusions should allow the solution of the more general problem of when one homomorphism between connected quandles factors through another.  It would also come close to solving the problem in the case where the map through which the first given map is to factor has a connected source, but an arbitrary target, as the orbit decomposition theorems of \cite{Ehrman, Nelson} reduce that problem to understanding maps into the connected subquandles arising from the iterated order decomposition.

The fully general problem could also be attacked in stages according to the ``depth'' of the orbit decompositions of the quandles involved (the number of iterations of orbit decomposition needed to reach connnected subquandles), either with or without the surjectivity condition.  The present paper solves the surjective case for depth 0.  The previous paragraph describes an attack on the problem at depth 0, but surjectivity dropped.  Another natural special case to attack next would be factorization conditions for surjective maps between quandles of depth less than or equal to 1.

\end{document}